\documentclass{amsart}
\usepackage{amssymb,amsmath,amsthm,enumerate,amscd,graphicx,color}
\usepackage[usenames,dvipsnames,svgnames,table]{xcolor}
\usepackage{pdfsync}
 \usepackage[colorlinks=true, pdfstartview=FitV, linkcolor=blue, citecolor=blue, urlcolor=blue]{hyperref}
\newcommand{\thmref}[1]{Theorem~\ref{#1}}

\newcommand{\lemref}[1]{Lemma~\ref{#1}}
\newcommand{\eqnref}[1]{~(\ref{#1})}

\newcommand{\germ}{\mathfrak}
\newtheorem{thm}{Theorem}[section]
\newtheorem{lem}[thm]{Lemma}

\theoremstyle{definition}


\numberwithin{equation}{section}

\usepackage{amssymb}
\begin{document}


\title{}

\title{On the Universal Central Extension of Hyperelliptic Current Algebras.}
\author{ Ben Cox}
\address{Department of Mathematics \\
University of Charleston \\
66 George St.  \\
Charleston, SC 29424, USA}\email{coxbl@cofc.edu}

 \begin{abstract}  Let $p(t)\in\mathbb C[t]$ be a polynomial with distinct roots and nonzero constant term.   We describe, using Fa\'a de Bruno's formula and Bell polynomials, the universal central extension in terms of generators and relations for the hyperelliptic current Lie algebras $\mathfrak g\otimes R$ whose coordinate ring is of the form $R=\mathbb C[t,t^{-1},u\,|\, u^2=p(t)]$.  
 \end{abstract}
\date{}
\thanks{ The author would like to thank the Mittag-Leffler Institute for its hospitality and support where most of this paper was written.  Travel to the Mittag-Leffler Institute was partially supported by a Simons Collaborations Grant.}


\maketitle

\section{Introduction}  Let $p(t)\in\mathbb C[t]$ be a separable polynomial with nonzero constant term. 
In this paper, using a general result of C. Kassel, we explicitly describe the universal central extension of the Lie algebra $\mathfrak{g}\otimes R$ where $R=\mathbb C[t,t^{-1},u\,|\, u^2=p(t)]$.  An algebraic curve of the form $\mathbb C[t,u\,|\, u^2=p(t)]$ is called a hyperelliptic curve if $\deg p>4$, so we called the universal central extension of $\mathfrak{g}\otimes R$, a hyperelliptic current algebra.   Modulo most of the center, this algebra is a particular example of a Krichever-Novikov current algebra (see \cite{MR902293},\cite{MR925072}, \cite{MR998426}).  See also the books \cite{MR2985911} and \cite{MR3237422} for more information and background about these algebras. 

In the 1990's M. Bremner explicitly described in terms of generators, relations and using certain families of polynomials (ultraspherical and Pollaczek), the structure constants for the universal central extension of algebras of the form $\mathfrak{g}\otimes R$ where $R=\mathbb C[t,t^{-1},u\,|\, u^2=p(t)]$ where $p(t)=t^2-2bt+1$ and $p(t)=t^3-2bt^2 +t$, $b\neq \pm 1$ (see \cite{MR1261553,MR1303073,MR1249871}). He determined more generally the dimension of the universal central extension for affine Lie algebras of the form $\mathfrak g\otimes R$ where $R$ is the ring of regular functions defined on an algebraic curve with any number of points removed. This was derived using C. KasselÕs result \cite{MR694130,MR772062}) where one knows that the center is isomorphic as a vector space to $\Omega_R/dR$.  The case where $p(t)=t^4-2bt^2+1$ was studied in \cite{MR2813377} and \cite{MR3090080} where particular cases of the associated Jacobi polynomials made their appearance. Thus the contents of the current paper includes as special cases most of the above work.  If the constant term of the polynomial is zero, then some small modifications of the present work will yield also an explicit description of the universal central extension.  This is left to the reader. We will later review some of the above cited work as needed.

It should also be mentioned that in joint work with R. Lu, X. Guo and K. Zhao we have described the center of the universal central extension of the Lie algebra $\text{Der}(R)$ for $R$ as above (\cite{CGLZ}) and for the $n$-point ring $R=\mathbb C[t,(t-a_1)^{-1},\dots, (t-a_n)^{-1}]$ for distinct complex numbers $a_i$ (\cite{MR3211093}).    Here also interesting families of polynomials arise in their description (for example the associated Legendre polynomials when $p(t)=t^{2k}-2bt^k+1$, $b\neq \pm 1$, $k\geq 1$). 

In our previous work (see \cite{MR2373448,MR3245847,MR2541818,MR2813377}), the authors we used such detailed information to obtain certain free-field realizations of the three point, four point, elliptic affine and DJKM algebras depending on a parameter $r=0,1$ that correspond to two different normal orderings. These later realizations are analogues of Wakimoto type realizations which have been used by Schechtman and Varchenko and various other authors in the affine setting to pin down integral solutions to the Knizhnik-Zamolodchikov differential equations (see for example \cite{MR1136712}, \cite{MR1138049}, \cite{MR2001b:32028}, \cite{MR1077959}). Such realizations have also been used in the study of the Drinfeld-Sokolov reduction in the setting of W-algebras and in E. Frenkel's and B. Feigin's description of the center of the completed enveloping algebra of an affine Lie algebra (see \cite{MR1309540, MR2146349}). 

Our main result \thmref{mainresult} is a description of the generators and relations for the hyperelliptic current algebra $\mathfrak g\otimes R$ in terms of polynomials arising from Fa\'a de Bruno's formula and Bell polynomials.  

In future work the author plans to use results of this paper to describe free-field realizations of the universal central extension of these hyperelliptic current algebras.  One can also see that the group of automorphisms of a commutative associate algebra $R$ over $\mathbb C$ induces an action of that group on $\Omega_R/dR$.  In future work we will also describe for certain algebras $R$ how $\Omega_R/dR$ decomposes into a direct sum of irreducible submodules under this action.

\section{Background, Fa\'a de Bruno's Formula and Bell Polynomials}
As C. Kassel showed the universal central extension of the current algebra $\mathfrak g\otimes R$ where $\mathfrak g$ is a finite dimensional Lie algebra defined over $\mathbb C$, is a vector space $\hat{\mathfrak g}=(\mathfrak g\otimes R)\oplus \Omega_R^1/dR$ where $\Omega_R^1/dR$ is the space of K\"ahler differentials modulo the exact forms $dR$ and where the Lie bracket is given by 
\begin{equation}
[x\otimes f,y\otimes g]:=[xy]\otimes fg +(x,y)\overline{f\,dg},\quad [x\otimes f,\omega]=0,\quad [\omega,\omega']=0,
\end{equation}
where $x,y\in \mathfrak g$, and $\omega,\omega'\in \Omega_R^1/dR$.   Here $(x,y)$ denotes the Killing form on $\mathfrak g$ and $\overline{f\,dg}$ denotes the residue class in $\Omega^1_R/dR$.

Consider the polynomial $p(t)=t^n+a_{n-1}t^{n-1}+\cdots+a_0$, where $a_i\in\mathbb C$ and $a_n=1$.  Fundamental to the description of $\hat{\mathfrak g}$ where $R=\mathbb C[t,t^{-1},u\,|\,u^2=p(t)]$, is the following:
\begin{thm}[\cite{MR1303073}, Theorem 3.4]   Let $R$ be as above.  The set
\begin{equation}
\{\overline{t^{-1}\,dt},\overline{t^{-1}u\,dt},\dots, \overline{t^{-n}u\,dt}\}
\end{equation}
forms a basis of $\Omega_R^1/dR$ (omitting $\overline{t^{-n}u\,dt}$ if $a_0=0$).
\end{thm}

A straightforward calculation shows that 
\begin{lem}[\cite{MR2813377}] If $u^m=p(t)$ and $R=\mathbb C[t,t^{-1},u\,|\,u^m=p(t)]$, then $\Omega^1_R/dR$, one has 
\begin{equation}\label{recursion1}
((m+1)n+im)t^{n+i-1}u\,dt\equiv -\sum_{j=0}^{n-1}((m+1)j+mi)a_jt^{i+j-1}u\,dt\mod dR.
\end{equation}
\end{lem}

Take from now on $m=2$ and 
$$
p(t)=(t-\alpha_1)\cdots (t-\alpha_n)=\sum_{i=0}^na_it^i
$$ where the $\alpha_i$ are distinct complex numbers and fix $R=\mathbb C[t,t^{-1},u\,|\,u^2=p(t)]$ so that $R$ is a regular ring. Motivated by \eqnref{recursion1} we let $P_{k,i}:=P_{k,i}(a_0,\dots, a_{n-1})$, $k\geq -n$, $-n\leq i\leq -1$ be the polynomials in the $a_i$ satisfying the recursion relations 
\begin{equation}\label{recursion2}
(2k+n+2)P_{k,i}=-\sum_{j=0}^{n-1}(3j+2k-2n+2)a_jP_{k-n+j,i}
\end{equation}
for $k\geq 0$ with the initial condition
$P_{k,i}=\delta_{k,-i}$, $-n\leq k\leq -1$.
 Now consider the formal power series 
\begin{equation}
P_i(z):=P_i(a_0,\dots, a_{n-1}, z):=\sum_{k\geq -n}P_{k,i}z^{k+n}=\sum_{n\geq 0}P_{k-n,i}z^k
\end{equation}
Then one can show that $P_i(z)$ must satisfy the differential equation
\begin{equation}
\frac{d}{dz}P_i(z)-\frac{Q(z)}{2z\bar P(z)}P_i(z)=\frac{R_i(z)}{2z\bar P(z)}
\end{equation}
where
\begin{gather*}
\bar P(z):=\sum_{j=0}^na_jz^{n-j},\quad Q(z):=z\bar P'(z)+(n-2)\bar P(z) 
\end{gather*}
and
\begin{gather*}
R_i(z):=\sum_{j=0}^n\left(\sum_{-j\leq k<0}(3j+2k-2n+2)a_jP_{k-n+j,i}z^{n+k}\right).
\end{gather*}
An integrating factor is 
$$
\mu(z)=\exp\int - \frac{Q(z)}{2z\bar P(z)} \,dz
=\frac{1}{z^{(n-2)/2}\sqrt{\bar P(z)}}.
$$
and so 
\begin{equation}
P_i(z):=z^{(n-2)/2}\sqrt{\bar P(z)}\int \frac{R_i(z)}{2z^{n/2}\bar P(z)^{3/2}} \,dz.
\end{equation}
The way we interpret the right hand hyperelliptic integral ($\bar P(0)=a_n=1\neq 0$) is to expand $R_l(z)/\bar P(z)^{3/2}$ in terms of a Taylor series about $z=0$ and then formally integrate term by term and multiply the result by series for $z^{(n-2)/2}\sqrt{\bar P(z)}$. 
 Let us explain this more precisely.

One can expand both $\sqrt{\bar P(z)}$ and $1/\bar P(z)^{3/2}$ using Bell polynomials and Fa\`a di Bruno's formula as follows.  
The Bell polynomials in the variables $z_1,z_2,z_3,\dots$ are defined to be 
\begin{align*}
B_{n,k}(z_1,\dots,z_{n-k+1}):=\sum \frac{n!}{l_1!l_2!\cdots l_{n-k+1}!}\left(\frac{z_1}{1!}\right)^{l_1} \cdots \left(\frac{z_{n-k+1}}{(n-k+1)!!}\right)^{l_{n-k+1}}
\end{align*}
where the sum is over $l_1+l_2+\cdots =k$ and $l_1+2l_2+3l_3+\cdots =n$ (see \cite{MR1502817}).

Now Fa\`a di Bruno's formula  \cite{dB} (discovered earlier by Arbogast \cite{Arbogast}) for the $n$-derivative of $f(g(x))$ is 
\begin{align*}
\frac{d^n}{dx^n}f(g(x))=\sum_{l=0}^nf^{(l)}(g(x))B_{n,l}(g'(x),g''(x),\dots, g^{(n-l+1)}(x)).
\end{align*}
Setting $f(x)=x^{-3/2}$, $g(x)=\bar P(x)$ we get 
\begin{equation}
f^{(n)}(x)=\frac{(-1)^n(2n+1)!!}{2^nx^{(2n+3)/2}}
\end{equation}
and $\bar P^{(k)}(0)=k!a_{n-k}$ so that 
\begin{align*}
\frac{d^n}{dx^n}f(g(x))|_{x=0}=\sum_{l=0}^n\frac{(-1)^ l(2 l+1)!!}{2^ l }B_{n, l}(a_{n-1},2a_{n-2},\dots,(n- l+1)!a_{l-1}).
\end{align*}
As a consequence 
\begin{align*}
\frac{1}{\bar P(z)^{3/2} }&=\sum_{n=0}^\infty \frac{1}{n!}\frac{d^n}{d z^n}f(g( z))|_{ z=0} z^n \\
&=\sum_{n=0}^\infty \frac{1}{n!} \left(\sum_{l=0}^n\frac{(-1)^ l(2 l+1)!!}{2^ l }B_{n, l}(a_{n-1},2a_{n-2},\dots,(n- l+1)!a_{l-1})\right) z^n ,
\end{align*}
and hence
\begin{equation}
P_n(a_0,\dots,a_{n-1})= \frac{1}{n!}  \sum_{l=0}^n\frac{(-1)^ l(2 l+1)!!}{2^ l }B_{n, l}(a_{n-1},2a_{n-2},\dots,(n- l+1)!a_{l-1}).
\end{equation}

Similarly for $\sqrt{\bar P}(z)$ we set $f(z)=\sqrt{z}$ so that 
$$
f^{(k)}(z)=\frac{(-1)^{k+1}(2k-3)!!}{2^kz^{(2k-1)/2}}
$$ 
for $n\geq 0$ 
and thus 
\begin{align*}
\sqrt{\bar P(z)}&=\sum_{n=0}^\infty \frac{1}{n!} \left(\sum_{l=0}^n\frac{(-1)^{l+1}(2l-3)!!}{2^l }B_{n, l}(a_{n-1},2a_{n-2},\dots,(n- l+1)!a_{l-1})\right) z^n.
\end{align*}
We now take \eqnref{recursion1}
\begin{equation}
2ka_0\overline{t^{k-1}u\,dt}= -\sum_{j=1}^{n}(3j+2k)a_j\overline{t^{k+j-1}u\,dt},
\end{equation}
 and write it as
 \begin{equation}\label{tneg}
-2(m-1)a_0\overline{t^{-m}u\,dt}= -\sum_{j=1}^{n}(3j-2m+2)a_j\overline{t^{-(m-j)}u\,dt}.
\end{equation}
which certainly is true for $m\geq n+1$.  This leads us to the recursion relation 
\begin{equation}\label{recursion3} 
-2(m-1)a_0Q_{m,i}=-\sum_{j=1}^{n}(3j -2m+2)a_jQ_{m-j,i}
\end{equation}
for $m\geq n+1$ with the same initial condition
$Q_{m,i}=\delta_{m,-i}$, $1\leq m\leq n$ and $-n\leq i\leq -1$. 
We then form the formal power series 
\begin{equation}
Q_i(z):=Q_i(a_0,a_1,\dots, a_{n-1},z)=\sum_{k\geq n+1}Q_{k-n,i}z^k=\sum_{k\geq 1}Q_{k,i}z^{k+n},
\end{equation}
for $1\leq i\leq n$. 
As above one can show that this formal series must satisfy

\begin{equation}
\frac{d}{dz}Q_i(z)-\frac{Q(z)}{2z P(z)}Q_i(z)=\frac{S_i(z)}{2z P(z)}
\end{equation}
where
\begin{gather}
P(z):=\sum_{j=0}^na_jz^{j},\quad  Q(z):=zP'(z)+2(n+1)P(z) ,\label{pqs}
\end{gather}
and
\begin{gather*}
   S_i(z):=-\sum_{m=1}^n\left(  \sum_{j=0}^{m-1} (3j -2m+2)a_jQ_{m-j,i}\right)z^{m+n} .\end{gather*}
Indeed
\begin{align*}
2z  P(z)\frac{d}{dz}Q_i(z)-Q(z)Q_i(z)
&=   \sum_{k\geq 1}\sum_{j=0}^n(2(k+n)-j-2n-2)a_jQ_{k,i}z^{j+k+n}\\
&= -   \sum_{k\geq 1}\sum_{j=0}^n(j-2k+2)a_jQ_{k,i}z^{j+k+n}\\
&=  - \sum_{m\geq n+1} \left( \sum_{j=0}^n(3j -2m+2)a_jQ_{m-j,i}\right)z^{m+n} \\
&\quad - \sum_{m=1}^n\left(  \sum_{j=0}^{m-1} (3j -2m+2)a_jQ_{m-j,i}\right)z^{m+n} \\
&=S_i(z).
\end{align*}

 An integrating factor is 
$$
\mu(z)=\exp\int - \frac{Q(z)}{2z  P(z)} \,dz
=\frac{1}{z^{n+1}\sqrt{ P(z)}},
$$
and so 
\begin{equation}
Q_i(z):=  z^{n+1}\sqrt{  P(z)} \int \frac{S_i(z)}{2 z^{n+2}P(z)^{3/2}} \,dz.
\end{equation}
\subsection{Example}  In \cite{MR3090080} we considered $p(t)=z^4-2c z^2+1$.
The polynomials $P_{-4,n}$ have generating series
\begin{align*}
P_{-4}(c,z)&=z\sqrt{1-2 c z^2+z^4}\int \frac{4cz^2-1}{z^2(z^4-2c z^2+1)^{3/2}}\, dz=\sum_{n=0}^\infty P_{-4,n}(c)z^n \\
&=1+z^4+\frac{4c}{5}z^6 +\frac{1}{35} \left(32 c^2-5\right) z^8+\frac{16}{105} c \left(8 c^2-3\right) z^{10} \\
 &\quad -\frac{\left(2048 c^4-1248 c^2+75\right) }{1155}z^{12}+O(z^{14})
\end{align*}
and
$P_{-4,n}(c)$ satisfy the following recursion:
\begin{equation}
(6+2k)P_{k+4}(c)
=4k cP_{k+2}(c)-2(k-3)P_{k}(c).
\end{equation}
The main result of \cite{MR3090080} is that these polynomials are particular examples of the non-classical associated Jacobi polynomials and they satisfy the following fourth order linear differential equation: 
\begin{thm}[\cite{MR3090080}, Theorem 3.1.1]
 The polynomials $P_n=P_{-4,n}$ satisfy the following differential equation:
 \begin{gather*} 
16 ( c^2-1)^2   P_{n}^{(iv)}+160 c  (c^2-1 )  P_{n}''' -8  (c^2  (n^2-4 n-46 )-n^2+4 n+22 )P_{n}''  \\
  -24 c ( n^2 - 4 n-6)P_{n}'  +(n-4)^2 n^2P_{n}=0.
\end{gather*}
\end{thm}
In the above cited work we also studied $P_{k,n}$ the remaining cases of $k=-3,-2,-1$.   In particular all of these families of polynomials are orthogonal with respect to some measure and satisfy an at most 4th order linear differential equation. 
\subsection{Example}  Let us take $p(t)=t^6-2bt^3+1$.  Then we get the recursion relation for the $P_{k,i}$'s as 
\begin{equation}
(2k+8)P_{k,i}=-\sum_{j=0}^{5}(3j+2k-10)a_jP_{k+j-6,i}=2b(2k-1)P_{k-3,i}-(2k-10)P_{k-6,i},
\end{equation}
for $k\geq 0$.
So that one can calculate by hand for example the first five nonzero nonconstant polynomials for $i=-1$; 
\begin{align*}
P_{2,-1}=\frac{b}{2},\quad P_{5,-1}=\frac{b^2}{2},\quad P_{8,-1}=\frac{1}{8} b \left(5
   b^2-1\right),\quad P_{11,-1}=\frac{1}{8} b^2 \left(7 b^2-3\right),\quad P_{14,-1}=\frac{1}{16}
   \left(21 b^5-14 b^3+b\right).
\end{align*}
In this setting of $p(t)$ we have
\begin{gather*}
R_{-1}(z)=6z^5,\quad R_{-2}(z)=4z^4,\quad R_{-3}(z)=2z^3,\quad R_{-4}(z)=6bz^5,  \\
R_{-5}(z)=10bz^4-2z,\quad R_{-6}(z)=14bz^3-4.
\end{gather*}
and as an example
\begin{align*}
P_{-1}(z)&=z^{2}\sqrt{z^6-2bz^3+1}\int \frac{3z^2}{(z^6-2bz^3+1)^{3/2}} \,dz \\
&=z^5+\frac{b z^8}{2}+\frac{b^2 z^{11}}{2}+\frac{1}{8} b \left(5 b^2-1\right)
   z^{14} +\frac{1}{8} b^2 \left(7 b^2-3\right) z^{17} \\
   &\quad+\frac{1}{16} \left(21
   b^5-14 b^3+b\right) z^{20}+\frac{1}{16} b^2 \left(33 b^4-30 b^2+5\right)
   z^{23}+O\left(z^{24}\right).
\end{align*}
Note in the integral we take as the constant of integration to be $0$.

One could conjecture that these polynomials $P_{k,i}$ are nonclassical orthogonal and satisfy $6$-th order linear differential equations with polynomial coefficients in $b$ for each fixed $i$, $-6\leq i\leq -1$ .   A similar conjecture can be made for polynomials $P_{k,i}$ arising from $p(t)=t^{2k}-2bt^k+1$ for $k\geq 4$ where the order of the differential equations would be $2k$.   We plan to investigate this in future work.
\section{Cocycles}

First we give an explicit description of the cocyles contributing to the {\it even} part of the hyperelliptic current algebra. 
Set
\begin{equation}
\omega_0:=\overline{t^{-1}\,dt},\enspace \omega_k:=\overline{t^{-k}u\,dt},\quad \text{ for }1\leq k\leq n\end{equation}
where we omit $\omega_n$ if $a_0=0$.

\begin{lem}[cf. \cite{MR1303073}, Prop. 4.2] \label{evenlem} 
 For $i,j\in\mathbb Z$ one has
\begin{equation}
\overline{t^i\,d(t^j)}= j \delta_{i+j,0}\omega_0
\end{equation}
and 
\begin{equation}
\overline{t^{i}u\,d(t^{j}u)}=\sum_{k=0}^n\left(j+\frac{1}{2}k\right)a_k\delta_{i+j,-k}\omega_0.
\end{equation}

\end{lem}
\begin{proof}  First observe that 
$$
2u\,du=d(u^2)=p'(t)\,dt=\sum_{i=0}^nia_it^{i-1}\,dt.
$$

 The second congruence then follows from
\begin{align*}
t^{i}u\,d(t^{j}u)&=jt^{i+j-1}u^2\,dt+t^{i+j}u\,du \\
&=jt^{i+j-1}\left(\sum_{k=0}^na_kt^k\right)\,dt+\frac{1}{2}t^{i+j}\sum_{k=0}^nka_kt^{k-1}\,dt \\
&=\sum_{k=0}^n\left(j+\frac{1}{2}k\right)a_kt^{i+j+k-1}\,dt .
\end{align*}
\end{proof}

For the odd part we have 
\begin{lem}[cf. \cite{MR1303073}, Prop. 4.2] \label{oddlem}  
 For $i,j\in\mathbb Z$ one has
\begin{equation}
\overline{t^{i}u\,d(t^{j})}=j\begin{cases} \sum_{k=1}^{n}P_{i+j-1,-k}\omega_{k} &\quad \text{ if }i+j\geq -n +1,\\
 \sum_{k=1}^{n}Q_{-i-j+1,-k}\omega_{k} &\quad \text{ if }i+j< -n+1.
 \end{cases}
\end{equation}
\end{lem}
\begin{proof}
We have 
\begin{equation}\label{eqn1}
(2(i+j)+n)t^{i+j-1}u\,dt\equiv -\sum_{k=0}^{n-1}(3k+2(i+j)-2n))a_kt^{i+j-n+k-1}u\,dt\mod dR.
\end{equation}
and similarly
\begin{equation} \label{eqn2}
(2(i+j)+n)P_{i+j-1,i}=-\sum_{k=0}^{n-1}(3k+2(i+j)-2n)a_kP_{i+j-n+k-1,i}
\end{equation}

%
%
So now assume for $r\geq -n$ 
\begin{align*}
\overline{t^ru\,dt}
&=\sum_{k=0}^{n-1}P_{r,k-n}\omega_{n-k}
\end{align*}
(which certainly holds for $r=-n,\dots,-1$).
Then 
\begin{align*}
\overline{t^{r+1}u\,dt} &=-\sum_{k=0}^{n-1}\left(\frac{3k+2r-2n+4}{n+2r+4}\right)a_k\overline{t^{r-n+k+1}u\,dt} \\
&=-\sum_{l=0}^{n-1}\sum_{k=0}^{n-1}\left(\frac{3k+2r-2n+4}{n+2r+4}\right)a_kP_{r+1-n+k,l-n}\omega_{n-l} \\
&=\sum_{l=0}^{n-1}P_{r+1,l-n}\omega_{n-l}
\end{align*}
Then for $i+j\geq -n+1$  one has
\begin{equation}
\overline{t^{i}u\,d(t^j)}=j\overline{t^{i+j-1}u\,dt}=j\sum_{k=0}^{n-1}P_{i+j-1,k-n}\omega_{n-k} =j\sum_{k=1}^{n}P_{i+j-1,-k}\omega_{k}.
\end{equation}

So now assume for $r\geq 1$ 
\begin{align*}
\overline{t^{-r}u\,dt}
&=\sum_{k=0}^{n-1}Q_{r,k-n}\omega_{n-k}
\end{align*}
(which certainly holds for $r=1,\dots ,n$ as $Q_{m,i}=\delta_{m,-i}$, $1\leq m\leq n$ and $-n\leq i\leq -1$.).

For $r\geq n$ we have by \eqnref{recursion3}
 \begin{align*}
\overline{t^{-(r+1)}u\,dt} &= \sum_{j=1}^{n}\frac{(3j-2r)a_j}{2ra_0}\overline{t^{-(r+1-j)}u\,dt}.
 \\
&=\sum_{k=0}^{n-1}\sum_{j=1}^{n}\frac{(3j-2r)a_j}{2ra_0}Q_{r+1-j,k-n}\omega_{n-k} \\
&=\sum_{k=0}^{n-1}Q_{r+1,k-n}\omega_{n-k}.
\end{align*}

Then for $i+j< -n+1$ we have 
\begin{equation}
\overline{t^{i}u\,d(t^j)}=j\overline{t^{i+j-1}u\,dt}=j\sum_{k=0}^{n-1}Q_{-i-j+1,k-n}\omega_{n-k} =j\sum_{k=1}^{n}Q_{-i-j+1,-k}\omega_{k}.
\end{equation}

\end{proof}

\section{Main result}
\begin{thm}[cf. \cite{MR2373448}]\label{mainresult}  Suppose $a_0\neq 0$.   Let $\mathfrak g$ be a simple finite dimensional Lie algebra over the complex numbers with Killing form $(\,|\,)$ and define $\psi_{ij}(c)\in\Omega_R^1/dR$ by
\begin{equation}
\psi_{ij}(c)=  \begin{cases} \sum_{k=1}^{n}P_{i+j-1,-k}\omega_{k} &\quad \text{ if }i+j\geq -n +1,\\
 \sum_{k=1}^{n}Q_{-i-j+1,-k}\omega_{k} &\quad \text{ if }i+j< -n+1.
 \\
\end{cases}
\end{equation} 
The universal central extension of the hyperelliptic Lie algebra $\mathfrak g \otimes R$ is the $\mathbb Z_2$-graded Lie algebra 
$$
\widehat{\mathfrak g}=\widehat{\mathfrak g}^0\oplus \widehat{\mathfrak g}^1,
$$
where
$$
\widehat{\mathfrak g}^0=\left(\mathfrak g\otimes \mathbb C[t,t^{-1}]\right)\oplus \mathbb C\omega_{0},\qquad \widehat{\mathfrak g}^1=\left(\mathfrak g\otimes \mathbb C[t,t^{-1}]u\right)\oplus\bigoplus_{k=1}^{n}\left( \mathbb C\omega_{k}\right)$$
with bracket
\begin{align}
[x\otimes t^i,y\otimes t^j]&=[x,y]\otimes t^{i+j}+\delta_{i+j,0}j(x,y)\omega_0, \label{even1}\\ \notag \\
[x\otimes t^{i}u,y\otimes t^{j}u]&=[x,y]\otimes t^{i+j}p(t)  +\sum_{k=0}^n\left(j+\frac{1}{2}k\right)a_k\delta_{i+j,-k}\omega_0,  \label{even2} \\ \notag \\
[x\otimes t^{i}u,y\otimes t^{j}]&=[x,y]u\otimes t^{i+j}u+ j(x,y)\psi_{ij}(c). \label{odd}
\end{align}

\end{thm}

\begin{proof}
The identities \eqnref{even1} and \eqnref{even2} follow from \lemref{evenlem} whereas \eqnref{odd} follows from \lemref{oddlem}

\end{proof}

\section{}
 Any corrections or typos that are found will be updated on the math arXiv.

\def\cprime{$'$} \def\cprime{$'$} \def\cprime{$'$}
\def\cprime{$'$} \def\cprime{$'$} \def\cprime{$'$}

 \end{document}